\documentclass[a4paper, oneside, reqno, 12pt]{amsart}
\usepackage[utf8]{inputenc}
\usepackage{amsmath, amssymb, amsfonts}
\usepackage{fullpage}
\usepackage{textcomp, cmap, comment}
\usepackage{mathtools} 
\usepackage[nobysame, initials]{amsrefs}

\usepackage{comment}
\usepackage{color, xcolor}

\theoremstyle{plain}
\newtheorem{theorem}{Theorem}
\newtheorem{conjecture}[theorem]{Conjecture}

\newtheorem{lemma}[theorem]{Lemma}
\newtheorem{corollary}[theorem]{Corollary}
\theoremstyle{definition}
\newtheorem{problem}[theorem]{Problem}

\theoremstyle{remark}
\newtheorem{remark}[theorem]{Remark}

\DeclareMathOperator{\tr}{tr}
\DeclareMathOperator{\rank}{rank}

\usepackage{pgf,tikz, mathrsfs}
\usetikzlibrary{arrows}

\usepackage{hyperref}
\hypersetup{
	colorlinks   = true, %Colours links instead of ugly boxes
	urlcolor     = blue, %Colour for external hyperlinks
	linkcolor    = red, %Colour of internal links
	citecolor   = green %Colour of citations
}

\title[Almost-equidistant sets]{{On almost-equidistant sets - II}}
\author[A.~Polyanskii]{{A.~Polyanskii}}%
\address{Alexandr Polyanskii,
\newline\hphantom{iii} \href{https://mipt.ru/english/}{MIPT}, \href{http://cmcagu.ru/}{CMC ASU}, \href{http://iitp.ru/en/about}{IITP RAS}
}
\email{\href{mailto:alexander.polyanskii@yandex.ru}{alexander.polyanskii@yandex.ru}}
\urladdr{\url{http://polyanskii.com}}


\begin{document}

\thispagestyle{empty}

\begin{abstract}
A set in $\mathbb R^d$ is called \textit{almost-equidistant} if for any three distinct points in the set, some two are at unit distance apart. First, we give a short proof of the result of Bezdek and L\'{a}ngi claiming that an almost-equidistant set lying on a $(d-1)$-dimensional sphere of radius $r$, where $r<1/\sqrt{2}$, has at most $2d+2$ points. Second, we prove that an almost-equidistant set $V$ in $\mathbb R^d$ has $O(d)$ points in two cases: if the diameter of $V$ is at most $1$ or if $V$ is a subset of a $d$-dimensional ball of radius at most $1/\sqrt{2}+cd^{-2/3}$, where $c<1/2$. Also, we present a new proof of the result of Kupavskii, Mustafa and Swanepoel that an almost-equidistant set in $\mathbb R^d$ has $O(d^{4/3})$ elements.
\end{abstract}

\maketitle

\section{Introduction}
A set in $\mathbb R^d$ is called \textit{almost-equidistant} if among any three points in the set, some two are at unit distance apart. The natural conjecture~\cite{Pol17}*{Conjecture~12} claims that an almost-equidistant set in $\mathbb R^d$ has $O(d)$ points. 

Using an elegant linear algebraic argument, Rosenfeld~\cite{R90} proved that an almost-equidistant set on a $(d-1)$-dimensional sphere of radius $1/\sqrt{2}$ has at most $2d$ points. Note that the set of the vertices of two unit $(d-1)$-simplices lying on the same $(d-1)$-dimensional sphere of radius $1/\sqrt{2}$ is  almost-equidistant. Bezdek and L{\'a}ngi~\cite{BL00}*{Theorem~1} generalized Rosenfeld's approach and showed that an almost-equidistant set on a $(d-1)$-dimensional sphere of radius at most $1/\sqrt{2}$ has at most $2d+2$ points; this bound is tight because the vertices of two unit $d$-simplices inscribed in the same sphere form an almost-equidistant set. 

Balko, P{\'o}r, Scheucher, Swanepoel and Valtr~\cite{BPSSV16}*{Theorem~6} showed that an almost equidistant set in $\mathbb R^d$ has $O(d^{3/2})$ points. This bound was improved by the author~\cite{Pol17}*{Theorem~1} to $O(d^{13/9})$. Recently, Kupavskii, Mustafa, Swanepoel~\cite{KMS17} further improved to $O(d^{4/3})$. For more references we refer interested readers to~\cite{BPSSV16}*{Section~1}.

The first goal of this paper is to give a short proof of the result of Bezdek--L\'{a}ngi~\cite{BL00} using a lifting argument and the fact that an almost-equidistant set on a $(d-1)$-dimensional sphere has at most $2d$ points. The second aim is to confirm the conjecture in two cases: for almost-equidistant sets of diameter $1$ (see~\hyperref[section:diameter]{Section~\ref*{section:diameter}}) and for almost-equidistant sets lying in a $d$-dimensional ball of radius $1/\sqrt{2}+cd^{-2/3}$, where $c<1/2$ (see~\hyperref[section:sphere]{Section~\ref*{section:sphere}}). The third goal is to give a new proof of the upper bound $O(d^{4/3})$ for the size of an almost-equidistant set in $\mathbb R^d$ (see \hyperref[section:newproof]{Section~\ref*{section:newproof}}). Also, we discuss several open problems related to almost-equidistant sets (see \hyperref[section:discussion]{Section~\ref*{section:discussion}}).

\section{Preliminaries}
\label{section:preliminaries}
Let $\{\mathbf v_1,\dots, \mathbf v_n\}\subset \mathbb R^d$ be an almost-equidistant set. Consider the matrix
\begin{equation}
\label{equation:matrix_u}
	\mathbf U:=\|\mathbf v_i-\mathbf v_j\|^2+\mathbf I_n-\mathbf J_n,
\end{equation}
where $\mathbf J_n$ is the $n$-by-$n$ matrix of ones and $\mathbf I_n$ is the identity matrix of size $n$.
We need two simple facts proved in~\cite{Pol17}*{Corollary 4 and Lemma 5}. We join them in the following lemma.
\begin{lemma}\label{lemma:eigenvalues}
The matrix $\mathbf U$ satisfies the following two properties.

1. The equalities $\tr (\mathbf U)=\tr(\mathbf U^3)=0$ hold.

2. The matrix $\mathbf U$ has at most one eigenvalue larger than $1$ and at least $n-d-2$ eigenvalues equal to $1$.
\end{lemma}
We use \hyperref[lemma:eigenvalues]{Lemma~\ref*{lemma:eigenvalues}} combined with the following lemma several times.
\begin{lemma}
\label{lemma:mainlemma}
Let $\lambda_0, 1,\dots, 1,\lambda_1,\dots,\lambda_k$ be the eigenvalues of an Hermitian matrix $\mathbf A$ of size $n$, indexed in nondecreasing order. Assume $\tr(\mathbf A)=\tr(\mathbf A^3)=0$. 

1. If $\lambda_0=1$, then $n\leq 2k$. 

2. If $\lambda_0+\lambda_k\leq 0$, then $n\leq 2k$. 

3. If $n\geq 2k$ and $\lambda_0>1$, then 
\[
\lambda_0^3>\frac{(n-k)^3}{k^2}-(n-k-1).
\]
\end{lemma}
\begin{proof}
\textit{1.} Suppose to the contrary that $n>2k$. Introducing the notation $l:=n-2k$, we can rewrite the equations $\tr(\mathbf A)=\tr(\mathbf A^3)=0$ as
\begin{equation*}
\label{equation:sumofeigenvalues}
	\sum_{i=1}^k(-\lambda_i)=\sum_{i=1}^k(-\lambda_i)^3=k+l.
\end{equation*}
To finish the proof, we need Lemma~1 in~\cite{BL00}. For the sake of completeness we provide its proof here.
\begin{lemma}
	\label{lemma:bezdeklangi}
	Let $x_1,\dots, x_m$ be real numbers such that $x_i\geq -2$ for $i=1,\dots,m$. If $\sum_{i=1}^mx_i=(m+l)$, where $l\geq 0$, then
	\[
	\sum_{i=1}^m x_i^3\geq \frac{(m+l)^3}{m^2}.
	\]
Here the equality is possible if and only if the $x_i$ are equal to $1+l/m$.
\end{lemma}
\begin{remark}
Also, note that ${(m+l)^3}/{m^2}\geq m+3l.$
\end{remark}
\begin{proof}
Consider functions $f, g:[-2,+\infty)\to \mathbb R$ such that
\[
	f(x)=x^3 \text{ for any }x\geq-2,\ 
    g(x)=\begin{cases}
    3x-2, &\text{ for any }-2\leq x\leq 1,\\
    x^3,  &\text{ for any }1\leq x.
    \end{cases}
\]
For $-2\leq x\leq 1$ we have $g(x)\leq f(x)$ because in this range $g(x)$ has the value of a tangent line to $f(x)$ at $x=1$ and the second point of the intersection of that tangent line and $f(x)$ is at $x=-2$. Further, $g(x)$ is a convex function in the range $-2\leq x$. By Jensen's Inequality, we obtain
\begin{gather*}
	\sum_{i=1}^m x_i^3=\sum_{i=1}^m f(x_i) \geq\sum_{i=1}^m g(x_i) \geq m g\left(\frac{\sum_{i=1}^mx_i}{m}\right)\\=m\left(\frac{\sum_{i=1}^m x_i}{m}\right)^3=\frac{(m+l)^3}{m^2}\geq m+3l.
\end{gather*}
The equality case we leave as an exercise.
\end{proof}

Clearly, \hyperref[lemma:bezdeklangi]{Lemma~\ref*{lemma:bezdeklangi}} implies that
\[
	\sum_{i=1}^k(-\lambda_i)^3\geq k+3l,
\]
and this is a contradiction with the inequality $l>0$.

\textit{2.} Suppose to the contrary that $n \geq 2k+1$. Introducing the notation $l:=n-2k-1$, we can rewrite the equalities $\tr(\mathbf A)=\tr(\mathbf A^3)=0$ as
\begin{equation}
\label{equation:sumofeigenvaluesnew}
	\sum_{i=1}^{k-1} (-\lambda_i)+ (-\lambda_0-\lambda_k)=\sum_{i=1}^{k-1}(-\lambda_i)^3-\lambda_0^3-\lambda_k^3=k+l.
\end{equation}
By \hyperref[lemma:bezdeklangi]{Lemma~\ref*{lemma:bezdeklangi}}, we obtain
\begin{equation}
\label{equation:sumofmodifiedeigenvalues}
	\sum_{i=1}^{k-1}(-\lambda_i)^3+(-\lambda_0-\lambda_k)^3\geq k+3l.
\end{equation}
The second equality in~\eqref{equation:sumofeigenvaluesnew} implies
\[
	k+l+(-\lambda_0-\lambda_k)^3+\lambda_0^3+\lambda_k^3\geq k+3l,
\]
and so
\begin{equation}
\label{equation:case2}
-3\lambda_0\lambda_k(\lambda_0+\lambda_k)\geq 2l.
\end{equation}
Since $\lambda_0>0$, $\lambda_k<0$, $\lambda_0+\lambda_k\leq 0$ and $l\geq 0$, we get $\lambda_0+\lambda_k=0$ and $l=0$. By~\hyperref[lemma:bezdeklangi]{Lemma~\ref*{lemma:bezdeklangi}}, if $\lambda_1+\lambda_k=0$ then we have a strict inequality in~\eqref{equation:sumofmodifiedeigenvalues}, and thus we get a strict inequality in~\eqref{equation:case2}. This implies a contradiction with the equality $l=0$.

\textit{3.} Clearly, we can rewrite the equalities $\tr (\mathbf A)=\tr(\mathbf A^3)=0$ as
\[
\sum_{i=1}^{k}(-\lambda_i)=\lambda_0+n-k-1> n-k
\]
and
\[
 \lambda_0^3=\sum_{i=1}^{k}(-\lambda_i)^3-(n-k-1).
\]
Since $n\geq 2k$, \hyperref[lemma:bezdeklangi]{Lemma~\ref*{lemma:bezdeklangi}} implies that
\[
\lambda_0^3 > \frac{(n-k)^3}{k^2}-(n-k-1).
\]
\end{proof}
\begin{corollary}
\label{corollary:useful}
If the matrix $\mathbf U$ does not have an eigenvalue larger than $1$, then $n\leq 2d+4$.
\end{corollary}
\begin{proof}
The proof easily follows from Lemmas~\ref{lemma:eigenvalues} and~\ref{lemma:mainlemma} (\textit{case 1}).
\end{proof}
We need the following technical lemma only in the proof of the fact that an almost-equidistant set has $O(d^{4/3})$ points; see \hyperref[theorem:kms17_4/3]{Theorem~\ref*{theorem:kms17_4/3}}.
\begin{lemma}
\label{lemma:end}
If $\{\mathbf w_0, \mathbf w_1,\dots,\mathbf w_k\}$ is an almost-equidistant set in $\mathbb R^d$ such that
\begin{equation}
\label{equation:radius}
\left|\|\mathbf w_i\|^2-1/2\right|\leq x
\end{equation}
for $0\leq i \leq k$ and a positive $x$, then
\[
\left| \sum_{1\leq i\leq k}\left(\|\mathbf w_0-\mathbf w_i\|^2-1\right)\right| \leq c \left(d^{1/2}+dx^{1/2}+dx\right)
\]
for some positive constant $c$ independent of $d$ and $x$.
\end{lemma}
\begin{proof} 
We may assume without loss of generality that $\|\mathbf w_0-\mathbf w_i\|\ne 1$ for $1\leq i\leq k$. Since the set is almost-equidistant, the points $\mathbf w_1,\dots, \mathbf w_k$ form a regular unit $(k-1)$-simplex, and so $k\leq d+1\leq 2d$.

We use the following theorem (see~\cite{DM94}*{Theorem~1}) several times.
\begin{theorem}\label{theorem:centermass}
	Let $X = \{\mathbf x_1,\dots, \mathbf x_n\}$ and $Y = \{\mathbf y_1, \dots, \mathbf y_n\}$ be two point-sets in $\mathbb R^d$. Then 
	\[
	\sum_{1\leq i,j\leq n} \|\mathbf x_i-\mathbf y_j\|^2 = \sum_{1\leq i<j\leq n} \|\mathbf x_i - \mathbf x_j\|^2 + \sum_{1\leq i< j\leq n}\|\mathbf y_i- \mathbf y_j\|^2 + n^2\|\mathbf x -\mathbf y\|^2,
	\]
	where $\mathbf x$ and $\mathbf y$ are the barycenters of $X$ and $Y$, respectively, that is, 
	\[
	\mathbf x =(\mathbf x_1 + \cdots + \mathbf x_n)/n,\ \mathbf y = (\mathbf y_1 + \cdots +\mathbf y_n)/n. 
	\]
\end{theorem}

\hyperref[theorem:centermass]{Theorem~\ref*{theorem:centermass}} applied to $\{\mathbf w_0,\dots, \mathbf w_0\}$ and $\{\mathbf w_1,\dots, \mathbf w_k \}$ implies
\begin{equation}
\label{eq:ineq1}
\sum_{1\leq i\leq k} \|\mathbf w_0-\mathbf w_i\|^2=\frac{k-1}{2} + k\|\mathbf w_0-\mathbf o'\|^2=\frac{k-1}{2}+k\|\mathbf w_0\|^2+k\|\mathbf o'\|^2+O(1)\cdot k\|\mathbf w_0\|\|\mathbf o'\|,
\end{equation}
where $\mathbf o'$ is the center of the simplex $\mathbf w_1\dots \mathbf w_k$. \hyperref[theorem:centermass]{Theorem~\ref*{theorem:centermass}} applied to $\{\mathbf o,\dots, \mathbf o\}$ and $\{\mathbf w_1,\dots,\mathbf w_k\}$, where $\mathbf o$ is the origin, and combined with \eqref{equation:radius} yields
\begin{equation}
\label{eq:ineq3}	
k\|\mathbf o'\|^2=\sum_{1\leq i\leq k}\|\mathbf w_i\|^2-\frac{k-1}{2}=\frac{k}{2}-\frac{k-1}{2}+O(1)\cdot kx=\frac{1}{2}+O(1)\cdot kx.
\end{equation}
By \eqref{equation:radius} and \eqref{eq:ineq3}, we obtain
\begin{gather*}
\eqref{eq:ineq1}= \frac{k-1}{2}+\frac{k}{2}+\frac{1}{2}+ O(1)\cdot kx+O(1)\cdot k(1+x^{1/2})(k^{-1/2}+x^{1/2})\\
=k+O(1)\cdot(k^{1/2}+kx^{1/2}+kx).
\end{gather*}
The fact $k\leq 2d$ finishes the proof of the lemma.
\end{proof}
\section{A simple proof of the result of Bezdek and L\'{a}ngi}
\label{section:simpleproof}
\begin{theorem}[Bezdek, L\'{a}ngi, 1999]
\label{theorem:bezdek-langi}
If an almost equidistant-set lies on a $(d-1)$-dimensional sphere $S$ of radius $r$, where $r \leq 1/\sqrt{2}$, then it has at most $2d+2$ points.
\end{theorem}
\begin{proof}
Assume that $S$ is embedded in $\mathbb R^{d+1}$ and its center $\mathbf o$ and a point $\mathbf o'\in\mathbb R^{d+1}$ are such that the line $\mathbf o' \mathbf o$ is orthogonal to the affine hull of $S$ and $|\mathbf o'-\mathbf o|^2 = 1/2-r$. Now, we get that the almost-equidistant set lies on the $d$-dimensional sphere of radius $1/\sqrt{2}$ with center $\mathbf o'$. By Rosenfeld's theorem~\cite{R90}, the size of the almost-equidistant sets is at most $2(d+1)$.
\end{proof}
\section{Almost-equidistant diameter sets}
\label{section:diameter}
A subset of $\mathbb R^d$ is called an \textit{almost-equidistant diameter set} if it is almost-equidistant and has diameter $1$; see \cite{Pol17}*{Proposition~13 and Problem~14}. The next theorem is about the maximal size of such sets.
\begin{theorem} 
\label{theorem:diameter} An almost-equidistant diameter set in $\mathbb R^d$ has at most $2d+4$ points.
\end{theorem}
\begin{proof}
Clearly, the matrix $\mathbf U$ (see \eqref{equation:matrix_u}) for an almost-equidistant diameter set in $\mathbb R^d$ has non-positive entries. \hyperref[lemma:eigenvalues]{Lemma~\ref*{lemma:eigenvalues}} and \hyperref[corollary:useful]{Corollary~\ref*{corollary:useful}} imply that we can assume without loss of generality that $\mathbf U$ has exactly one eigenvalue $\lambda$ larger than $1$. 

We need the following week form of the \hyperref[theorem:perron]{Perron--Frobenius Theorem}; see ~\cites{Per1907,Fro1912}.
\begin{theorem}[Perron--Frobenius Theorem]
	\label{theorem:perron}
If an $n$-by-$n$ matrix has non-negative entries, then it has
a non-negative real eigenvalue, which has maximum absolute value among all eigenvalues.
\end{theorem}
By the \hyperref[theorem:perron]{Perron--Frobenius Theorem}, the matrix $\mathbf U$ has a negative eigenvalue $\lambda'$ such that $|\lambda'|\geq\lambda$. Therefore, Lemmas~\ref{lemma:eigenvalues} and~\ref{lemma:mainlemma} (\textit{case 2}) imply that the almost-equidistant diameter set has at most $2d+2$ points.
\end{proof}
\section{Almost-equidistant sets in small balls}
\label{section:sphere}
\begin{theorem}
\label{theorem:smallspheres}
Let $0 \leq c_0 < 1/2$ be a fixed constant. If an almost-equidistant set lies in a $d$-dimensional ball of radius $\sqrt{1/2 + c_0/(d+1)^{2/3}}$, then its cardinality is $O(d)/(1-2c_0)$.
\end{theorem}
\begin{proof}
Let $\{\mathbf v_1,\dots, \mathbf v_n\}$ be an almost-equidistant set in $\mathbb R^d$ such that $n\geq 2d+2$ and $\|\mathbf v_i\|^2=1/2+c_i$ for $1\leq i\leq n$, where $c_i\leq c_0/(d+1)^{2/3}=r$.

By \hyperref[corollary:useful]{Corollary~\ref*{corollary:useful}} and \hyperref[lemma:eigenvalues]{Lemma~\ref*{lemma:eigenvalues}}, without loss of generality we can assume that the matrix $\mathbf U$ (see~\eqref{equation:matrix_u}) for this set has exactly one eigenvalue $\lambda$ larger than $1$.  Hence Lemmas~\ref{lemma:eigenvalues} and~\ref{lemma:mainlemma} (\textit{case 3}) for $\mathbf U$ yield
\[
	\lambda^3> \frac{(n-d-1)^3}{(d+1)^2}-(n-d-2).
\]
To finish the proof, we need the following special case of \hyperref[theorem:sumofHermitian]{Weyl's Inequality}~\cite{W1912}*{Theorem~1}.
\begin{theorem}[Weyl's~Inequality]
\label{theorem:sumofHermitian}
If $\alpha, \beta$ and $\gamma$ are the largest eigenvalues of Hermitian matrices $\mathbf A, \mathbf B$ and $\mathbf A+\mathbf B$ respectively, then $\gamma \leq \alpha+\beta$.
\end{theorem}
Clearly, by~\eqref{equation:matrix_u}, we have
\[
\mathbf U=(c^t j+j^tc)-2\langle \mathbf v_i, \mathbf v_j\rangle +\mathbf I,
\]
where $c=(c_1,\dots, c_n)$ and $j=(1,\dots,1)$. Obviously, the largest eigenvalues of the matrices $c^t j$ and $j^t c$ do not exceed $nr$.
Therefore, from \hyperref[theorem:sumofHermitian]{Weyl's Inequality} and positivity of the Gram matrix $\langle \mathbf v_i, \mathbf v_j \rangle$, we conclude that $\lambda\leq 2nr+1$. This forces
\begin{equation}
\label{equation:interesting}	
\left(2nr+1\right)^3\geq \lambda^3> \frac{(n-d-1)^3}{(d+1)^2}-(n-d-2),
\end{equation}
and so
\[
	(2c_0x+o(1))^3\geq (x-1)^3-(x-1-o(1)),
\]
where $x=n/(d+1)$. Hence we get $x=O(1)/(1-2c_0)$, and thus $n=O(d)/(1-2c_0)$.
\end{proof}
\begin{remark}
Interestingly, an almost-equidistant set lying in a $d$-dimensional ball of radius $1/\sqrt{2}$ has at most $2d+4$ points. Indeed, in that case $r$ is equal to $0$, and hence \eqref{equation:interesting} yields $n<2d+2$ if $\mathbf U$ has an eigenvalue larger then $1$. Therefore, by \hyperref[corollary:useful]{Corollary~\ref*{corollary:useful}}, we get $n\leq 2d+4$.
\end{remark}
\section{Almost-equidistant sets: general case}
\label{section:newproof}

\begin{theorem}[Kupavskii, Mustafa, Swanepoel, 2018]
\label{theorem:kms17_4/3}
The size of an almost-equidistant set in $\mathbb R^d$ is $O(d^{4/3})$.
\end{theorem}
\begin{proof}
Assume that $\{\mathbf v_1,\dots, \mathbf v_n\}$ is an almost-equidistant set in $\mathbb R^d$. Let
\begin{equation}
\label{equation:definition_of_f}
f:=\max_{i=1,\dots,n}\Bigl\{\Bigl| \sum_{j=1}^n\left(\|\mathbf v_i-\mathbf v_j\|^2-1\right)\Bigr|\Bigr\}.
\end{equation}
Assuming that $\sum_{i=1}^n \mathbf v_i$ is the origin $\mathbf o$, we easily get
\[
-f\leq n\|\mathbf v_i\|^2 +\sum_{j=1}^n\|\mathbf v_j\|^2- n\leq f
\]
for $1\leq i\leq n$. Summing up these inequalities for $1\leq i\leq n$, we obtain
\[
-f/2\leq \sum_{i=1}^n\|\mathbf v_i\|^2 - n/2\leq f/2.
\]
The last two inequalities implies
\begin{equation}
\label{equation:radius2}
-3f/(2n)\leq \|\mathbf v_i\|^2-1/2\leq 3f/(2n)
\end{equation}
for $1\leq i\leq n$. By \eqref{equation:definition_of_f}, \eqref{equation:radius2} and \hyperref[lemma:end]{Lemma~\ref*{lemma:end}}, we have
\[
	f\leq O(1)\cdot\left(d^{1/2}+d(f/n)^{1/2}+d(f/n)\right).
\]
Therefore, we obtain either $n=O(d)$ or
\begin{equation}
\label{equation:assumption}
f=O(1)\cdot \left(d^{1/2}+d^2/n\right).
\end{equation}

Suppose to the contrary that $n\geq Cd^{4/3}$, where $C$ is a positive constant. Hence, by~\eqref{equation:radius2} and~\eqref{equation:assumption}, the almost-equidistant set lies in a ball of radius $1/\sqrt{2}+O(d^{-2/3})/C^2$. Hence if $C$ is big enough, then we can apply \hyperref[theorem:smallspheres]{Theorem~\ref*{theorem:smallspheres}} to the almost-equidistant set, and thus $n=O(d)$, a contradiction.
\end{proof}

\begin{remark}
It is possible to prove \hyperref[theorem:kms17_4/3]{Theorem~\ref*{theorem:kms17_4/3}} using \hyperref[lemma:end]{Lemma~\ref*{lemma:end}} combined with the Gershgorin Circle Theorem~\cite{G31}; see~\cite{Pol17}*{Section 3.2} for a similar argument.
\end{remark}
\section{Discussion}
\label{section:discussion}
\subsection{Almost-equidistant diameter sets.}
A graph $(V,E)$ is called a \textit{diameter graph} if its vertex set $V\subseteq \mathbb R^d$ is a set of points of diameter $1$ and a pair of vertices forms an edge if they are at unit distance apart. Of course, the set of vertices of two cliques in a diameter graph is an almost-equidistant diameter set. For instance, in \cite{KP17}*{the last paragraph of Section~3} there is given an example of diameter graph in $\mathbb R^d$ consisting of two cliques \textit{without common vertices} such that they have $d+1$ and $\lfloor \frac{d+1}{2}\rfloor$ vertices respectively. We believe that the vertex set of this diameter graph has the maximal size among almost-equidistant diameter sets in $\mathbb R^d$. 
\begin{conjecture}
	\label{conjecture:1}
	An almost-equidistant diameter set in $\mathbb R^d$ has at most $\left\lfloor\frac{3(d+1)}{2}\right\rfloor$ points.
\end{conjecture}
 There is the following conjecture~\cite{Kal15}*{Conjecture~5.5} that arose in the context of study of cliques in diameter graphs.
\begin{conjecture}[Schur]
	\label{conjecture:schur}
	Let $S_1$ and $S_2$ be two unit simplices in $\mathbb R^d$ forming a set of diameter $1$ such that they have $k$ and $m$ vertices respectively. Then $S_1$ and $S_2$ share at least $\min \{0, k+2m-2d-2\}$ vertices for $k\geq m$.
\end{conjecture}
Clearly, this conjecture is closely related to \hyperref[conjecture:1]{Conjecture~\ref{conjecture:1}}. Note that \hyperref[conjecture:schur]{Conjecture~\ref*{conjecture:schur}} was confirmed in two special (but not trivial!) cases: $(k, m,d)=(d,d,d)$, where $d\geq 2$, in~\cite{KP17} and $(k, m, d)=(5, 3,4)$ in~\cite{KP217}.

\subsection{Two-distant almost-equidistant sets.}  A subset of $\mathbb R^d$ is called a \textit{two-distant set} if there are only two distances formed by any two distinct points of the set. The following question seems to be interesting.
\begin{problem}
	\label{problem:problem_two-distance}
	What is the largest cardinality of a set in $\mathbb R^d$ that is two-distant and almost-equidistant at the same time?
\end{problem}
Consider a two-distant almost-equidistant set $V\subseteq\mathbb R^d$ with distances $1$ and $a$ between points of $V$, where $a>1$; the case $a<1$ is not interesting because of \hyperref[theorem:diameter]{Theorem~\ref{theorem:diameter}}. \hyperref[lemma:eigenvalues]{Lemma~\ref*{lemma:eigenvalues}} and \hyperref[corollary:useful]{Corollary~\ref*{corollary:useful}} imply that we may assume without loss of generality that $\mathbf U$ for $V$ has exactly one eigenvalue $\lambda$ larger than $1$. Note that the matrix $\mathbf U/(a^2-1)$ is the adjacency matrix of a triangle-free graph. By \hyperref[lemma:eigenvalues]{Lemma~\ref*{lemma:eigenvalues}}, \hyperref[problem:problem_two-distance]{Problem~\ref*{problem:problem_two-distance}} is reduced to following question.
\begin{problem}
	What is the minimal rank of the matrix $\mathbf A-\lambda_2 \mathbf I_n$, where $\mathbf A$ is the adjacency matrix of a triangle-free graph on $n$ vertices and $\lambda_2$ is its second largest eigenvalue provided that $\lambda_2>0$?
\end{problem}
It is worth pointing out that for every positive integer $n$ there are a triangle-free graph on $n$ vertices and an eigenvalue $\lambda$ of the adjecancy matrix $\mathbf A$ of the graph such that $\rank(\mathbf A-\lambda \mathbf I)=O(n^{3/4})$ and $\lambda>0$; see~\cite{CPR00}*{the proof of Theorem~5, pages 94--95}. 
\section*{Acknowledgement}
We are grateful to \href{http://www.zilin.one}{Zilin Jiang} for bringing the \hyperref[theorem:perron]{Perron--Frobenius Theorem} to our attention. The author was supported by the Russian Foundation for Basic Research, grant \textnumero~18-31-00149 mol\_a.

\bibliographystyle{amsplain}
\bibliography{biblio}
\end{document}